\documentclass[a4paper,11pt]{article}
\usepackage{geometry}
\usepackage{amsmath, amsthm, amssymb}
\usepackage{bm}
\usepackage{graphicx}
\usepackage{fullpage}
\usepackage[colorlinks,linkcolor=blue,citecolor=blue,urlcolor=magenta,linktocpage,plainpages=false]{hyperref}\usepackage{xcolor}
\usepackage{xspace}
\usepackage{url}
\usepackage{libertine}
\usepackage{libertinust1math}
\usepackage[ruled,vlined]{algorithm2e}
\usepackage{authblk}

\newtheorem{lemma}{Lemma}
\newtheorem{thm}{Theorem}
\newtheorem{cor}{Corollary}
\newtheorem{proposition}{Proposition}
\newtheorem{fact}{Fact}

\newtheorem{definition}{Definition}

\newcommand{\Z}{\mathbb{Z}}

\newcommand{\R}[0]{\ensuremath \mathbb{R}}
\newcommand{\e}{\varepsilon}
\newcommand{\ip}[2]{\langle #1, #2 \rangle}

\newcommand{\ones}{\bm{1}}

\newcommand{\mom}[1]{{\left\vert\kern-0.25ex\left\vert\kern-0.25ex\left\vert #1 \right\vert\kern-0.25ex\right\vert\kern-0.25ex\right\vert}}

\newcommand{\conv}{\text{conv}}
\newcommand{\BBdepth}{\text{BBdepth}}
\newcommand{\BBrank}{\text{BBrank}}

\newcommand{\BBhardness}{\text{BBhardness}}

\newcommand{\T}{\mathcal{T}}

\DeclareMathOperator{\argmax}{argmax}

\title{Lower Bounds on the Size of General Branch-and-Bound Trees}
\author[1]{Santanu S. Dey\thanks{santanu.dey@isye.gatech.edu}}
\author[1]{Yatharth Dubey\thanks{yatharthdubey7@gatech.edu, (732) 439-2933}}
\author[2]{Marco Molinaro\thanks{molinaro@inf.puc-rio.br}}
\affil[1]{School of Industrial and Systems Engineering, Georgia Institute of Technology}
\affil[2]{Computer Science Department, PUC-Rio}
\date{}

\begin{document}
\maketitle
\begin{abstract}
A \emph{general branch-and-bound tree} is a branch-and-bound tree which is allowed to use general disjunctions of the form $\pi^{\top} x \leq \pi_0 \,\vee\, \pi^{\top}x \geq \pi_0 + 1$, where $\pi$ is an integer vector and $\pi_0$ is an integer scalar, to create child nodes. We construct a packing instance, a set covering instance, and a Traveling Salesman Problem instance, such that any general branch-and-bound tree that solves these instances must be of exponential size. We also verify that an exponential lower bound on the size of general branch-and-bound trees persists even when we add Gaussian noise to the coefficients of the cross-polytope, thus showing that a polynomial-size ``smoothed analysis'' upper bound is not possible. The results in this paper can be viewed as the branch-and-bound analog of the seminal paper by Chv\'atal et al.~\cite{chvatal1989cutting}, who proved lower bounds for the Chv\'atal-Gomory rank.
\end{abstract}

\newpage

\section{Introduction}

Solving combinatorial optimization problems to optimality is a central object of study in Operations Research, Computer Science, and Mathematics. One way to solve a combinatorial optimization problem is to model it as an integer program (IP),
namely a problem of the form 
	\begin{align}
		\begin{split}
		\max ~& \ip{c}{x}\\
		\textrm{s.t.}~& Ax \le b \\
		& x \in \Z^n 
		\end{split} \tag{IP} \label{IP}
	\end{align}
	
and then use an {IP} solver. The branch-and-bound algorithm, invented by Land and Doig in~\cite{land1960automatic}, is the underlying algorithm implemented in all modern  state-of-the-art MILP solvers. 

As is well-known, the branch-and-bound algorithm searches the solution space by recursively partitioning it. The progress of the algorithm is monitored by maintaining a tree. Each node of the tree corresponds to a linear program (LP) solved, and in particular, the root-node corresponds to the LP relaxation of the integer program (i.e., the where  the constraint $x \in \Z^n$ in \eqref{IP} is removed). After solving the LP corresponding to a node, the feasible region of the LP  is partitioned into two subproblems (which correspond to the child nodes of the given node), so that the fractional optimal solution of the LP is not included in either subproblem, but any integer feasible solution contained in the feasible region of the LP is included in one of the two subproblems. This is accomplished by adding an inequality of the form $\pi^{\top} x \leq \pi_0$ to the first subproblem and the inequality $\pi^{\top}x \geq \pi_0 + 1$ to the second subproblem (these two inequalities are referred as a disjunction), where $\pi$ is an integer vector and $\pi_0$ is an integer scalar (see Figure \ref{fig:bb}). The process of partitioning at a node stops if (i) the LP at the node is infeasible, or (ii) the LP's optimal solution is integer feasible, or (iii) the LP's optimal objective function value is worse than an already known integer feasible solution.  These three conditions are sometimes referred to as the rules for pruning a node. The algorithm terminates when there are no more ``open nodes'' to process, that is all nodes have been pruned. A branch-and-bound algorithm is completely described by fixing a rule for partitioning the feasible region at each node and a rule for selecting which open node should be solved and branched on next.  If the choice of $\pi$ is limited to being the canonical basis vectors $e_j = (0,\ldots,0,1,0,\ldots,0)$ (with the 1 in the $j$-th position), then we call such an algorithm a \emph{simple} branch-and-bound, and without such a restriction on $\pi$ we call the algorithm a \emph{\textbf{general}} branch-and-bound. See ~\cite{wolsey1999integer,conforti2014integer} for more discussion on branch-and-bound and for general background on integer programming.  

\begin{figure}[ht]\label{fig:bb}
\centering
\includegraphics[width=0.45\textwidth]{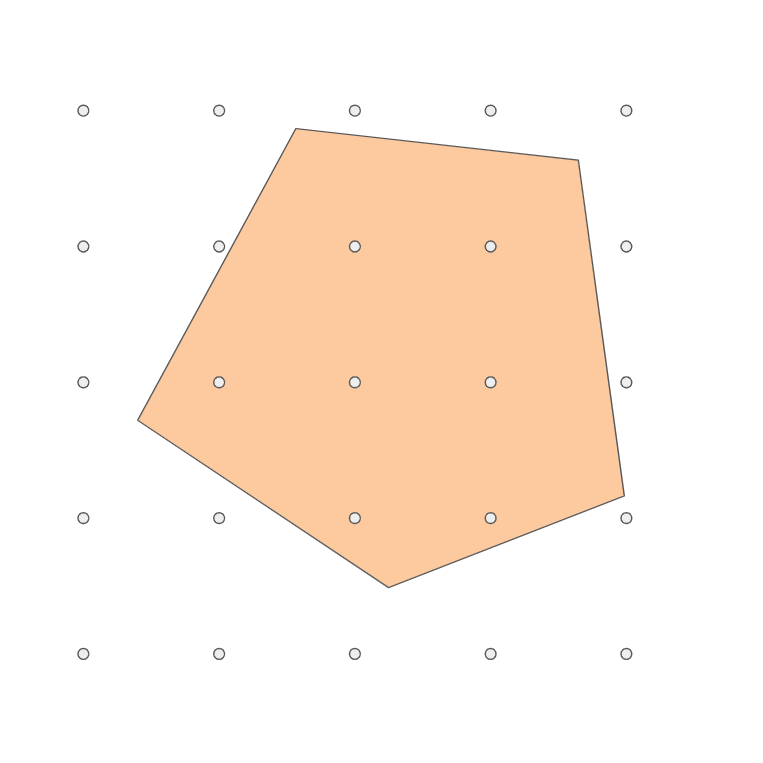}
\includegraphics[width=0.45\textwidth]{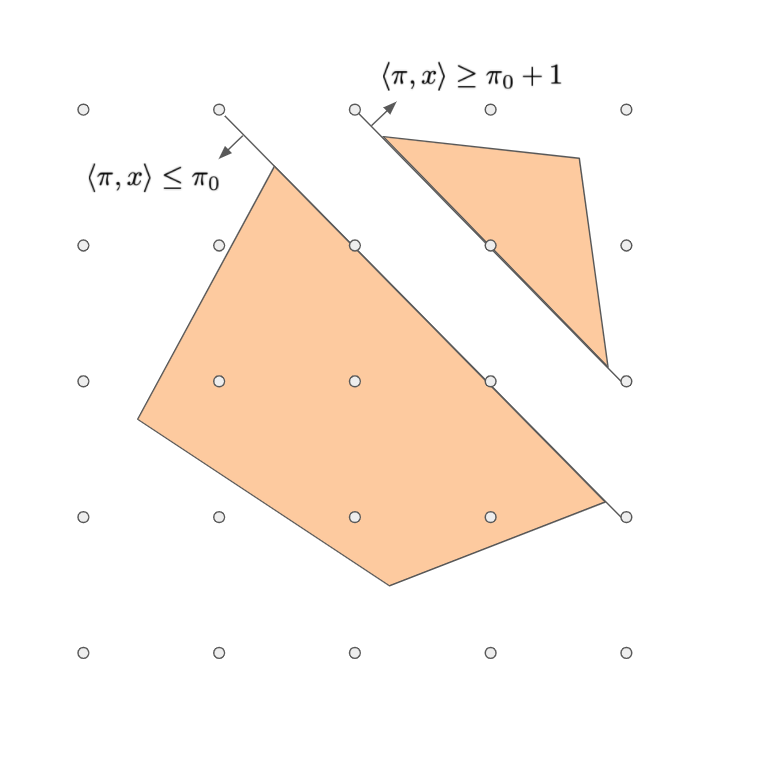}
\caption{The \textbf{left} picture is the initial polytope and the \textbf{right} picture demonstrates the two subproblems as a result of branching on the disjunction $\ip{\pi}{x} \leq \pi_0 \vee \ip{\pi}{x} \geq \pi_0 + 1$. }
\end{figure}

Here we take interest in the size of general branch-and-bound trees. This is because instances requiring an exponential number of nodes to solve using general branch-and-bound are likely to also be challenging for MILP solvers. We hope that this study can provide some intuition on when and why solvers struggle with a MILP instance and how to formulate heuristics to combat these bottlenecks.


\subsection{Known bounds on the size of branch-and-bound trees}


\paragraph{Upper bounds on the size of branch-and-bound trees and ``positive'' results.} In 1983, Lenstra~\cite{lenstra1983integer} showed that integer programs can be solved in polynomial time in fixed dimension. This algorithm can be viewed as a general branch-and-bound algorithm that uses tools from the geometry of numbers, in particular the lattice basis reduction algorithm~\cite{lenstra1982factoring}, to decide on $\pi$ for partitioning the feasible region. Pataki~\cite{pataki2010basis} proved that most random packing IPs (i.e., where $A$ and $b$ in \eqref{IP} are non-negative) can be solved at the root-node using a partitioning scheme similar to the one proposed by Lenstra~\cite{lenstra1983integer}. It has been observed that using such general partitioning rules can result in significantly smaller trees than using a simple branch-and-bound for some instances~\cite{aardal2000market,gerardGeneralBranching}, but most commercial solvers use the latter.
Recently, we showed~\cite{dey2021branch} that for certain classes of random integer programs the simple branch-and-bound tree has polynomial size (number of nodes), with good probability. See also~\cite{borst2020integrality} for nice extensions of this direction of results. Beame et al.~\cite{beame2018stabbing} recently studied how branch-and-bound can give good upper bounds for certain SAT formulas. 

\paragraph{Lower bounds on the size of branch-and-bound trees and connections to the size of cutting-plane algorithms.} 
Jeroslow~\cite{jeroslow1974trivial} and Chv\'atal~\cite{ chvatal1980hard} present examples of integer programs where every \emph{simple} branch-and-bound algorithm for solving them has an exponential-size tree.
 However, these instances can be solved with small (polynomial-size) \emph{general} branch-and-bound trees; see Yang et al.~\cite{yang2021multivariable} and Basu et al.~\cite{basu2020complexity2}. Cook et al. \cite{cook1990complexity} present a TSP instance that requires exponential-size branch-and-\emph{cut} trees that uses simple branching (recall that branch-and-cut is branch-and-bound where one is allowed to add cuts to the intermediate LPs). Basu et al.~\cite{basu2020complexity} compare the performance of branch-and-bound with the performance of cutting-plane algorithms, providing instances where one outperforms the other and vice-versa. In another paper, Basu et al.~\cite{basu2020complexity2} compare branch-and-bound with branch-and-cut, providing instances where branch-and-cut   solves the instance in exponentially fewer nodes than branch-and-bound. They also present a result showing that the sparsity of the disjunctions can have a large impact on the size of the branch-and-bound tree required to solve a given problem. 
Beame et al. \cite{beame2018stabbing} asked as an open question whether there are superpolynomial lower bounds for general branch-and-bound algorithm. Dadush and Tiwari~\cite{dadush2020complexity} settled this in the affirmative. In particular, they show that any general branch-and-bound tree that proves the integer infeasibility of the so-called cross-polytope in $n$-dimensions has at least $\frac{2^n}{n}$ leaf nodes. They also note that the cross-polytope has an exponential number of defining inequalities, a fact crucially used in their proof, and pose the open question of whether there is such an exponential lower bound for a polytope described by a polynomial number of defining inequalities.

Concurrent to the development of our work, Fleming et al. \cite{fleming2021power} showed a fascinating relationship between general branch-and-bound proofs and cutting-plane proofs using Chv\'atal-Gomory (CG) cutting-planes:

\begin{thm}[Theorem 3.7 from \cite{fleming2021power}]
Let $P \subseteq [0,1]^n$ be an integer-infeasible polytope and suppose there is a general branch-and-bound proof of infeasibility of size $s$ and with maximum coefficient $c$. Then there is a CG proof of infeasibility of size at most
$$s(cn)^{\log s}.$$
\end{thm}

The following simple corollary allows one to infer exponential lower bounds for branch-and-bound trees for polytopes for which we have exponential lower bounds for CG proofs. 

\begin{cor}\label{cor:sp_eq_cp}
Let $P \subseteq [0,1]^n$ be an integer-infeasible polytope such that any CG proof of integer-infeasibility of $P$ (see Definition \ref{def:abstract_bb}) has length at least $L$. Then any general branch-and-bound proof of integer-infeasibility of $P$ with maximum coefficient $c$ has size at least 
$$L^{\frac{1}{1 + \log(cn)}}.$$
\end{cor}

The above result makes progress in answering questions raised in Basu et al.~\cite{basu2020complexity} related to the comparison between the size of general branch-and-bound trees and the size of  CG proofs.
Moreover, Pudlak \cite{pudlak1997lower} and Dash \cite{dash2005exponential} provide exponential lower bounds for CG proofs for the ``clique vs. coloring'' problem, which is of note since this problem is defined by only polynomially-many inequalities. Thus, Corollary~\ref{cor:sp_eq_cp} taken together with results in \cite{pudlak1997lower} and \cite{dash2005exponential} also settles the question raised in Dadush and Tiwari~\cite{dadush2020complexity} as long as the maximum coefficient in the disjunctions used in the tree is bounded by a polynomial in $n$.


\subsection{Contributions of this paper and relationship to existing results}

\paragraph{Contributions.} 
We construct an instance of packing-type and a set-cover instance such that any general branch-and-bound tree that solves these instances must be of exponential (with respect to the ambient dimension) size. We note that the packing and covering instances are described using an exponential number of constraints, and so unfortunately this does not settle the question raised by Dadush and Tiwari~\cite{dadush2020complexity}. We also present a simple proof that any branch-and-bound tree proving the integer infeasibility of the cross-polytope in $n$ dimensions must have $2^n$ leaves. We then extend this result to give (high-probability) exponential lower bounds for perturbed instances of the cross-polytope where independent Gaussian noise is added to the entries of the constraint matrix. To our knowledge this is the first result that shows that a ``smoothed analysis''~\cite{roughgardenBook} polynomial upper bound on the size of branch-and-bound trees is not possible.
Finally, we show an exponential lower bound on the size of any general branch-and-bound tree for the Traveling Salesman Problem (TSP).


\paragraph{Comparison to previous results.}
We now discuss our results in the context of the recent landscape, in particular with the results of \cite{dadush2020complexity} and \cite{fleming2021power}. 
\begin{enumerate}
\item \emph{New problems with exponential lower bounds on the size of general branch-and-bound tree:} As mentioned earlier, recently Dadush and Tiwari~\cite{dadush2020complexity} provided the first  exponential lower bound on the size of general branch-and-bound tree for the cross-polytope. The other additional result, Corollary \ref{cor:sp_eq_cp} from \cite{fleming2021power}, only implies branch-and-bound lower bounds for polytopes for which we already have CG hardness. These come few and far between in the existing literature, and these instances are often a bit artificial; see \cite{pudlak1997lower} and \cite{dash2005exponential}. In contrast, in this paper we provide lower bounds for the size of general branch-and-bound tree for packing and set-cover instances, which are more natural combinatorial problems than those mentioned above.

\item \emph{Improved quality of bounds:} Dadush and Tiwari~\cite{dadush2020complexity} show that any branch-and-bound proof of infeasibility of the cross-polytope has at least $\frac{2^n}{n}$ leaves. We improve on this result by providing a simple proof that any such proof of infeasibility must have $2^n$ leaves. 

Chvatal et al. \cite{chvatal1989cutting} provide a $\frac{1}{3n}2^{n/8}$ lower bound on CG proofs for TSP. Combined with Corollary~\ref{cor:sp_eq_cp}, this can be used to show a lower bound of $2^{O\left({\frac{n}{\log c n}}\right)}$ for branch-and-bound trees for TSP using maximum coefficient $c$ for the disjunctions. We are able to achieve a stronger lower bound of $2^{\Omega(n)}$. 
    
    \item \emph{Removing the dependence on the maximum coefficient size used in the branch-and-bound proof:} The bound given in Corollary \ref{cor:sp_eq_cp} depends on the maximum coefficient size used in the branch-and-bound proof. In \cite{fleming2021power}, the authors mention that they ``view this as a step toward proving [branch-and-bound] lower bounds (with no restrictions on the [coefficient sizes])''. Our results satisfy this property, as none of the bounds presented in this work depend on the coefficients of the inequalities of the general branch-and-bound proof. 
\end{enumerate}

Finally, the results presented here can be easily combined with Theorem 1.14 of \cite{basu2020complexity2} to apply to branch-and-cut proofs. In particular, our results imply exponential lower bounds, for the polytopes shown here, on the size of branch-and-cut proofs that are allowed to branch on split disjunctions and employ any cutting-plane paradigm that is ``not sufficiently different'' from split disjunctive cuts (see \cite{basu2020complexity2} for details).

\subsection{Roadmap and notation} 

Since this paper focuses on lower bounds for general branch-and-bound trees (obviously implying lower bounds for simple branch-and-bound tree), we drop the term ``general'' for the rest of the paper. The paper is organized as follows. In Section~\ref{sec:defn} we present the necessary definitions. In Section~\ref{sec:tech_lemmas}, we present key reduction results that allow transferring lower bounds on the size of branch-and-bound trees from one optimization problem to another.  In Section~\ref{sec:combinatorial}, we present a lower bound on the size of branch-and-bound trees for packing and set covering instances.  In Section~\ref{sec:cross_poly_hard}, we present a lower bound on the size of branch-and-bound trees for the cross-polytope and some other related technical results. In Section~\ref{sec:perturbed_cp}, we show that even after adding Gaussian noise to the coefficients of the cross-polytope, with good probability branch-and-bound still requires an exponentially large tree to prove infeasibility. Finally, in Section~\ref{sec:tsp}, we use results from Section~\ref{sec:cross_poly_hard} and Section~\ref{sec:tech_lemmas} to provide an exponential lower bound on the size of branch-and-bound trees for solving TSP instances.

For a positive integer $n$, we denote the set $\{1, \dots, n\}$ as $[n]$. When the dimension is clear from context, we use the notation $\ones$ to be a vector whose every entry is $1$. Let $C$ be a set of linear constraints of the form $(\pi^i)^{\top} x \le \pi^i_0,~ \forall i \in [m]$. Then let $\{x : C\}$ denote the set of all $x \in [0,1]^n$ such that all of the constraints $C$ are valid for $x$ (i.e. the polytope defined by the set of constraints $C$). Note that for a subset of these constraints $B \subseteq C$, it holds that $\{x : B\} \supseteq \{x : C\}$. Also note that for two sets of constraints $B,C$, it holds that $\{x : B \cup C\} = \{x : B\} \cap \{x : C\}$. Given a set $S$, we denotes its convex hull by $\textup{conv}(S)$. Given a polytope $P \subseteq \mathbb{R}^n$, we denote its integer hull, that is the set $\textup{conv}(P \cap \Z^n)$, as $P_I$. We call $P$ integer-infeasible if $P \cap \Z^n = \emptyset$.

\section{Abstract branch-and-bound trees and notions of hardness}\label{sec:defn}

{In order to present lower bounds on the size of branch-and-bound (BB) trees, we simplify our analysis by removing two typical condition assumed in a BB algorithm -- (i) the requirement that the partitioning into two subproblems (which correspond to the child nodes of the given node) is done in  such a way that the optimal LP solution of the parent node is not included in either subproblem, and that (ii) branching is not done on pruned nodes. By removing these conditions, we can talk about a branch-and-bound tree independent of the underlying polytope -- it is just a full binary tree (that is, each node has $0$ or $2$ child nodes). The root-node has an empty set of \emph{branching constraints}. If a node has two child nodes, these are obtained by applying some disjunction $\pi^{\top} x \leq \pi_0 \,\vee\, \pi^{\top}x \geq \pi_0 + 1$, where each of the child nodes adds one of these constraints to its set of branching constraints together with all the branching constraints of the parent node.} Note that here $\pi$ is an \emph{integer} vector and $\pi_0$ an \emph{integer}; we call such disjunctions \emph{legal}. Note that proving lower bounds on the size of such BB trees that solves a given integer program certainly gives a lower bound on the size of BB trees that in addition require (i) and (ii).
Finally, note that since a BB tree is a full binary tree, the total number of nodes of a BB tree with $N$ leaf-nodes is $2N - 1$.

\begin{definition}\label{def:abstract_bb}
Given a branch-and-bound tree $\mathcal{T}$, applied to a polytope $P \subseteq \R^n$, and a node $v$ of the tree:
\begin{itemize}\vspace{-8pt}
\item We denote the number of nodes of the branch-and-bound tree $\mathcal{T}$ by $|\mathcal{T}|$. This is what is termed \emph{the size} of this tree.
\item We denote by $C_v$ the set of branching constraints of $v$ (as explained above, these are the constraints added by the branch-and-bound tree along the path from the root-node to $v$).
\item We call the feasible region defined by the LP relaxation $P$ and the branching constraints at node $v$ the \emph{atom} of this node, i.e., $P \cap \{x : C_v\}$ is the atom corresponding to $v$.
\item We let $\mathcal{T}(P)$ denote the {union of the atoms corresponding to the} leaves of $\mathcal{T}$ when run on polytope $P$, i.e., $\mathcal{T}(P) = \bigcup_{v \in \text{leaves}(\mathcal{T})} (P \cap \{x : C_v\})$.
\item For any $x^* \in P \setminus P_I$, we say that $\T$ \emph{separates} $x^*$ from $P$ if $x^* \not \in \conv(\T(P))$. 
\item Given a vector $c \in \mathbb{R}^n$, we say $\mathcal{T}$ solves $\max_{x \in P \cap \Z^n} \ip{c}{x}$ if for all the leaf nodes $v$ of $\T$, one of the following three conditions hold: (i) the atom of $v$ is empty, (ii) there exists at least one 
optimal solution of the linear program $\max_{x \,\in\, \textrm{atom of }v}\ip{c}{x}$ that is integral, or (iii) $\max_{x \,\in\, \textrm{atom of }v}\ip{c}{x}$ is at most the objective function value of another atom whose optimal solution is integral. 

If $P \cap \Z^n = \emptyset$, note that (ii) and (iii) are not possible, and in this case we use the term ``proves integer-infeasibility'' instead of ``solves'' the problem.
\end{itemize} 
\end{definition}


Given a polytope $P \subseteq \mathbb{R}^n$, we define its \emph{BB hardness} as
$$\textit{BBhardness}(P) = \max_{c \in \mathbb{R}^n} \left(\min \left\{| \mathcal{T}| : \mathcal{T} \textup{ solves } \max_{x \in P \cap \Z^n} \ip{c}{x}\right \} \right).$$

Our goal for most of this paper is to provide lower bounds on the BB hardness of certain polytopes. 


To get exponential lower bounds on BB hardness for some $P$, we will often present a particular point $x^* \in P \setminus P_I$ such that any $\T$ that separates $x^*$ from $P$ must have exponential size. We formalize this below. 

\begin{definition}[BBdepth]
Let $P \subseteq \R^n$ be a polytope and consider any $x^* \in P \setminus P_I$. Let $\T$ be a smallest BB tree that separates $x^*$ from $P$. Then, define $\text{BBdepth}(x^*, P)$ to be $|\T|$.
\end{definition}

\begin{definition}[BBrank]
Define $\BBrank(P) = \max_{x \in P \setminus P_I} \BBdepth(x, P)$.
\end{definition}

\begin{lemma}[BB rank lower bounds BB hardness]\label{lem:rankhardness}
Let $P \subseteq \R^n$ be a polytope. Then, there exists $c \in \R^n$ such that any BB tree solving $\max_{x \in P \cap \Z^n} \ip{c}{x}$ must have size at least $\BBrank(P)$, that is $\BBhardness(P) \geq \BBrank(P)$.
\end{lemma}

\begin{proof}
Let $x^* \in \argmax_{x \in P \setminus P_I} \BBdepth(x, P)$, so that $\BBrank(P) = \BBdepth(x^*, P)$. Since $x^*$ does not belong to the convex set $P_I$, by the hyperplane separation theorem~\cite{barvinok} there exists $c$ with the separation property $\ip{c}{x^*} > \max_{x \in P_I} \ip{c}{x}$. By choice of $x^*$, for any BB tree $\T$ with $|\T| < \BBrank(P)$ it holds that $x^* \in \conv(\T(P))$.  Then such tree $\T$ must have a leaf whose optimal LP solution has value at least $\ip{c}{x^*} > \max_{x \in P_I} \ip{c}{x}$, and therefore must still  not be pruned, showing that $\T$ does not solve $\max_{x \in P \cap \Z^n} \ip{c}{x}$.
\end{proof}

We now show that, under some conditions, the reverse of this kind of relationship also holds. 
We will use this reverse relationship to prove the BB hardness of optimizing over an integer feasible polytope given the BB hardness of proving the infeasibility of another ``smaller" polytope.

\begin{lemma}[Infeasibility-to-optimization]\label{lem:inf_to_opt}
Let $P \subseteq \R^n$ be a polytope and $\ip{c}{x} \le \delta$ be a facet defining inequality of $P_I$ {that is not valid for $P$}. Assume that the affine hull of $P$ and $P_I$ are the same. Then, {there exists $\e_0 > 0$ such that for every $\e \in (0, \e_0]$} 
$$\BBrank(P) \geq \BBhardness(\{x \in P : \ip{c}{x} \geq \delta + \e\}).$$
\end{lemma}

Before we can present the proof of Lemma \ref{lem:inf_to_opt}
we require a technical lemma from~\cite{dash2015relative}. The full-dimensional case $L = \R^n$ is Lemma 3.1 of~\cite{dash2015relative}, and the general case follows directly by applying it to the affine subspace $L$.


\begin{lemma}[\cite{dash2015relative}]\label{lem:height}
Consider an affine subspace $L \subseteq \R^n$ and a hyperplane $H = \{x \in \R^n :\ip{c}{x} = \delta\}$ that does not contain $L$. Consider $\dim(L)$ affinely independent points $s^1, s^2 \dots, s^{\dim(L)}$ in $L \cap H$. Consider $\delta' > \delta$ and let $G$ be a bounded and non-empty subset of $L \cap \{x \in \mathbb{R}^n: \ip{c}{x} \geq \delta'\}$. Then there exists a point $x$ in $\bigcap_{g \in G} \conv(s^1, \dots, s^{\dim(L)}, g)$ satisfying the strict inequality $\ip{c}{x} > \delta$. 
\end{lemma}

%

\begin{proof}[Proof of Lemma \ref{lem:inf_to_opt}]
	Let $L$ be the affine hull of $P_I$. Then there exist $d := \dim(P_I) = \dim(L)$ affinely independent vertices of $\{x \in P_I : \ip{c}{x} = \delta\}$. Let $s^1, \dots, s^d$ be $d$ such affinely independent vertices and note that since they are vertices of $P_I$, they are all integral. {Let $\e_0 := (\max_{x \in P} \ip{c}{x}) - \delta$, and notice that since the inequality $\ip{c}{x} \le \delta$ is not valid for $P$ we have $\e_0 > 0$. Then for any $\e \in (0,\e_0]$,} let $G := \{x \in P : \ip{c}{x} \geq \delta + \e\}$, {which is then non-empty}. Also notice that $G$ is a bounded set, since $P$ is bounded. Let $N := \BBhardness(G)$.

Let $\T$ be a BB tree such that $|\T| < N$. Then we have that $\T(G) \neq \emptyset$, that is, there exists $x^*(\T) \in \T(G)$. In particular $x^*(\T) \in G$. Moreover, since $G \subseteq P$, we have $\T(G) \subseteq \T(P)$  (see Lemma~\ref{lem:monotonicity} in the next section for a formal proof of this), and so we have $x^*(\T) \in  \T(P)$. Also note that since $s^1, \dots, s^d \in P \cap \mathbb{Z}^n$, we have that these points also belong to $\T(P)$. Thus, $$\conv\left(s^1, \dots, s^d, x^*(\T)\right)  \subseteq \conv(\T(P)).$$

Now applying Lemma~\ref{lem:height},  with $\delta' = \delta + \e$, we have that there exists $x^*$ such that
\begin{eqnarray}\label{eq:height_eq}
x^* \in \bigcap_{ \T: |\T| < N} \conv\left(s^1, \dots, s^d, x^*(\T)\right) \subseteq \bigcap_{ \T: |\T| < N} \conv(\T(P))
\end{eqnarray}
and such that {$\ip{c}{x^*} > \delta$}. Clearly, $x^* \not \in P_I$, since $\ip{c}{x} \leq \delta$ is a valid inequality for $P_I$. Thus, since \eqref{eq:height_eq} implies $x^* \in \conv(\T(P))$ for all $\T$ with $|\T|< N$, we have that $\BBdepth(x^*, P) \geq N$ and consequently, $\BBrank(P) \geq N$.
\end{proof}

\section{Framework for BB hardness reductions}\label{sec:tech_lemmas}


{In this section we present key reduction results that allow transferring lower bounds on the size of BB trees from one optimization problem to another.} We begin by showing monotonicity of the operator $\mathcal{T} (\cdot)$.

\begin{lemma}[Monotonicity of leaves]\label{lem:monotonicity}
Let $Q \subseteq P \subseteq \R^n$ be polytopes. Then $\T(Q) \subseteq \T(P)$. 
\end{lemma}

\begin{proof}
For any leaf $v \in \T$, recall that $C_v$ is the set of branching constraints of $v$. Then $\T(Q) = \bigcup_{v \in \text{leaves}(\T)} (Q \cap \{x : C_v\}) = Q \cap \bigcup_{v \in \text{leaves}(\T)} \{x : C_v\} \subseteq P \cap \bigcup_{v \in \text{leaves}(\T)} \{x : C_v\} = \bigcup_{v \in \text{leaves}(\T)} (P \cap \{x : C_v\}) = \T(P)$.
\end{proof}

The following corollary follows easily from Lemma \ref{lem:monotonicity}. In particular, consider {a} smallest BB tree $\T$ that separates $x^*$ from $P$. By Lemma \ref{lem:monotonicity}, the same tree, when applied to $Q \subseteq P$, will not have $x^*$ in the convex hull of its leaves and therefore separates $x^*$ from $Q$.

\begin{cor}[Monotonicity of depth]\label{cor:monoton_depth}
Let $Q \subseteq P \subseteq \R^n$ be polytopes. Then for every $x^* \in (Q \setminus Q_I) \cap (P \setminus P_I) = Q \setminus P_I$ we have
$$\BBdepth(x^*, Q) \le \BBdepth(x^*, P).$$
\end{cor}

Inspired by the lower bounds for cutting-plane rank from~\cite{chvatal1989cutting}, we show that integral affine transformations conserve the hardness of separating a point via branch-and-bound, i.e. they conserve BBdepth. Then, we give a condition where BBrank is also conserved. These  will be used to obtain lower bounds in the subsequent sections. 

	We say that $f: \R^n \rightarrow \R^m$ is an \emph{integral affine function} if it has the form $f(x) = Cx + d$, where $C \in \Z^{m \times n}, d \in \Z^m$.

\begin{lemma}[Simulation for integral affine transformations]\label{lem:affine}
Let $P \subseteq \R^n$ be a polytope, $f : \R^n \rightarrow \R^m$ an integral affine function, and denote $Q := f(P) \subseteq \R^m$. Let $\hat{\T}$ be any BB tree. Then, there exists a BB tree $\T$ such that $|\T| = |\hat{\T}|$ and
$$f(\T(P)) \subseteq \hat{\T}(Q).$$
\end{lemma}

\begin{proof}
Let $f(x) = Cx + d$ with $C \in \Z^{m \times n}, d \in \Z^m$. Given a BB tree $\hat{\T}$, we construct a BB tree $\T$ with the desired properties as follows: $\T$ has the same nodes as $\hat{\T}$ but each branching constraint $\ip{a}{y} \leq b$ of $\hat{\T}$ is replaced by the constraint $\ip{C^T a}{x} \leq b - \ip{a}{d}$ in $\T$. 

First we verify that $\mathcal{T}$ only uses legal disjunctions: First note that $C^T a \in \mathbb{Z}^n$ and $ b - \ip{a}{d} \in \mathbb{Z}$. If a node of $\hat{\T}$ has $\ip{a}{y} \leq b \ \vee \  \ip{a}{y} \geq  b + 1$ as its disjunction, the corresponding node in $\T$ has the disjunction $\ip{C^T a}{x} \leq b - \ip{a}{d}~\vee~\ip{-C^Ta}{x} \leq -b  -1 - \ip{-a}{d}$ (notice $\ip{a}{y} \geq  b + 1 \,\equiv\, \ip{-a}{y} \leq  -b - 1$). Since the second term in the latter disjunction is equivalent to  $\ip{C^Ta}{x} \geq b  - \ip{a}{d} + 1$, we see that this disjunction is a legal one. 

	To conclude the proof, we show that $f(\T(P)) \subseteq \hat{\T}(Q)$. Let $S$ be the atom of a leaf $v$ of $\T$ and $\hat{S}$ be the atom of the corresponding leaf $\hat{v}$ of $\hat{\T}$. We show that for all $x \in S$, it must be that $f(x) \in \hat{S}$.  To see this, notice that if $x$ satisfies an inequality $\ip{C^T a}{x} \leq b - \ip{a}{d}$ then $f(x)$ satisfies $\ip{a}{f(x)} \leq b$:
$$\ip{a}{f(x)} = \ip{a}{Cx + d}= \ip{a}{Cx} + \ip{a}{d} = \ip{C^T a}{x} + \ip{a}{d} \leq b.$$ Since any $x \in S$ belongs to $P$ and satisfies all the branching constraints of the leaf $v$, this implies $f(x)$ belongs to $Q$ and satisfies all the branching constraints of the leaf $\hat{v}$, and hence belongs to the atom $\hat{S}$. {Therefore, $f(S) \subseteq \hat{S}$. Taking a union over all leaves/atoms then gives $f(\T(P)) \subseteq \hat{\T}(Q)$ as desired.} 
\end{proof}

\begin{cor}\label{cor:affine_infeas}
Let $P$, $Q$, and $f$ satisfy the assumptions of Lemma \ref{lem:affine}. Further, suppose $P$ and $Q$ are both integer-infeasible. Then, 
$$\BBhardness(Q) \geq \BBhardness(P).$$
\end{cor}

\begin{proof}
Let $\hat{\T}$ be the smallest BB tree such that $\hat{\T}(Q) = \emptyset$. Then, by Lemma \ref{lem:affine}, it must hold that $\T(P) = \emptyset$. The desired result follows.
\end{proof}

\begin{cor}\label{cor:affine_depth}
Let $P$, $Q$, and $f$ satisfy the assumptions of Lemma \ref{lem:affine}. Then for every $x^* \in \R^n$ such that $x^* \not \in P_I$ and $f(x^*) \not \in Q_I$, we have 
$$\BBdepth(f(x^*), Q) \geq \BBdepth(x^*, P).$$
\end{cor}

\begin{proof}
Let $\hat{\T}$ be a smallest BB tree that separates $f(x^*)$ from $Q$, and let $\T$ be a tree given by Lemma \ref{lem:affine}. Together with the fact that $f$ is affine, this implies that if $x \in \conv(\T(P))$ then $f(x) \in \conv(\hat{\T}(Q))$: there exists $x^1, ..., x^k \in \T(P) \text{ and } \lambda_1, ..., \lambda_k \in [0,1] \text{ {such that} } \sum_{i \in [k]}\lambda_i = 1$ {and} $x = \sum_{i \in [k]} \lambda_i x^i$. {Thus, }
	%
	\begin{align*}
&f(x) = C(\sum_{i \in [k]} \lambda_i x^i) + d = \sum_{i \in [k]} \lambda_i (Cx^i) +  \sum_{i \in [k]} \lambda_i d\\
&~~ = \sum_{i \in [k]} \lambda_i (Cx^i + d) = \sum_{i \in [k]} \lambda_i f(x^i) \in \conv(\hat{\T}(Q)),
	\end{align*}
where the last containment is by definition of $\T$. Since we know $f(x^*) \notin \conv(\hat{\T}(Q))$, this implies that $x^* \notin \conv(\T(P))$, namely $\T$ separates $x^*$ from $P$ as desired. 
\end{proof}

\begin{lemma}[Hardness lemma]\label{lem:hardness}
Let $P \subseteq \R^n$ and $T \subseteq \R^m$ be polytopes and $f : \R^n \rightarrow \R^m$ an integral affine function such that $f(P) \subseteq T$. Suppose $f$ is also one-to-one and $T \cap \Z^m \subseteq f(P \cap \Z^n)$. Then,
$$\BBrank(T) \geq \BBrank(P).$$
\end{lemma}

\begin{proof}
First we show that $x \not \in P_I$ implies $f(x) \not \in T_I$ by proving the contrapositive. Suppose $f(x) \in T_I$; then $\exists y^1, ..., y^k \in T \cap \Z^m \text{ and } \lambda_1, ..., \lambda_k \in [0,1] \text{ {such that} } \sum_{i \in [k]}\lambda_i = 1$ {and} $f(x) = \sum_{i \in [k]} \lambda_i y^i$. Since $T \cap \Z^m \subseteq f(P \cap \Z^n)$, for each $i$ there is $x^i \in P \cap \Z^n$ such that $y^i = f(x^i)$.
Then
$$f(x) = \sum_{i \in [k]} \lambda_i f(x^i) = \sum_{i \in [k]} \lambda_i (Cx^i + d) = C\sum_{i \in [k]} \lambda_i x^i + d = f(\sum_{i \in [k]} \lambda_i x^i),$$
and so $f(x)$ belongs to $f(P_I)$. Since $f$ is one-to-one, this implies that $x$ belongs to $P_I$, as desired. 

Now let $x^* = \argmax_{x \in P \setminus P_I} \BBdepth(x, P)$. Since $x^* \notin P_I$, by the above claim $f(x^*) \not \in T_I$. By assumption $f(P) \cap \Z^m \subseteq T \cap \Z^m$, and so $(f(P))_I \subseteq T_I$, and therefore $f(x^*) \not \in (f(P))_I$. Then by Corollary \ref{cor:affine_depth} we have $$\BBdepth(f(x^*), f(P)) \geq \BBdepth(x^*, P).$$
Since by assumption $f(P) \subseteq T$, $f(x^*) \in f(P)$, and  $f(x^*) \not \in T_I$, by Corollary \ref{cor:monoton_depth} we have $$\BBdepth(f(x^*), T) \geq \BBdepth(f(x^*), f(P)).$$ Putting it all together we get
$$\BBrank(T) = \max_{y \in T \setminus T_I} \BBdepth(y,T) \geq \BBdepth(f(x^*), T) \geq \BBdepth(f(x^*), f(P)) \geq \BBdepth(x^*, P) $$
$$= \max_{x \in P \setminus P_I} \BBdepth(x, P) = \BBrank(P),$$
which concludes the proof of the lemma. 
\end{proof}

In the rest of the paper, we will use Corollary~\ref{cor:affine_infeas}, Corollary \ref{cor:affine_depth} or Lemma~\ref{lem:hardness} together with some appropriate affine transformation to reduce the BB hardness of one problem to another. The three affine one-to-one functions we will use (and their compositions) are Flipping, Embedding, and Duplication as defined below. 

\begin{definition}[Flipping]\label{defn:flip}
We say $f : [0,1]^n \rightarrow [0,1]^n$ is a \emph{flipping operation} if it ``flips'' some coordinates, that is, there exists $J \subseteq [n]$ such that 
$$y = f(x) \implies y_i = \begin{cases}
x_i & \text{if } i \not \in J \\
1 - x_i & \text{if } i \in J
\end{cases}.$$
In other words, $f(x) = Cx + d$, where (recall $e_i$ is the $i$-th canonical basis vector)
$$C^i = \begin{cases}
e_i & \text{ if } i \not \in J \\
-e_i & \text{ if } i \in J
\end{cases}$$
$$d_i = \begin{cases}
0 & \text{ if } i \not \in J \\
1 & \text{ if } i \in J
\end{cases}.$$
\end{definition}

\begin{definition}[Embedding]\label{defn:embed}
We say $f : [0,1]^n \rightarrow [0,1]^{n + k}$ is an \emph{embedding operation} if
$$y = f(x) \implies y_i = \begin{cases}
x_i & \text{if } 1 \leq i \leq n \\
0 & \text{if } n < i \leq n + k_1 \\
1 & \text{if } n + k_1 < i \leq n + k
\end{cases},$$
for some $0 \leq k_1 \leq k$. In other words, $f(x) = Cx + d$, where
	\begin{align*}
C^i &= \begin{cases}
e_i & \text{ if } 1 \leq i \leq n \\
0 & \text{ otherwise }
\end{cases}\\
d_i &= \begin{cases}
1 & \text{ if } n+k_1 < i \leq n + k \\
0 & \text{ otherwise }
\end{cases}.
	\end{align*}
Note that we can always renumber the coordinates so that the additional coordinates with values $0$ or $1$ are interspersed with the original ones and not grouped at the end. 
\end{definition}

\begin{definition}[Duplication]\label{defn:dup}
Consider a $k$-tuple of coordinates $(j_1,...,j_k)$ that are not necessarily distinct, where $j_i \in \{1,...,n\}$ for $i = 1,...,k$. We say that $f : [0,1]^n \rightarrow [0,1]^{n + k}$ is a \emph{duplication operation} using this tuple if
$$y = f(x) \implies y_i = \begin{cases}
x_i & \text{if } 1 \leq i \leq n \\
x_{j_{i - n}} & \text{if } n < i \leq n + k
\end{cases}.$$ 
Further, let $J_j = \{i \in \{1,...,k\} : y_{n+i} = x_j\}$ be the indices of $y$ that are duplicates of $x_j$. Then, in other words, $f(x) = Cx$ where
$$C^i = \begin{cases}
e_i & \text{ if } 1 \leq i \leq n \\
e_1 & \text{ if } i - n \in J_1 \\
\vdots \\
e_n & \text{ if } i - n \in J_n
\end{cases}$$
\end{definition}

\section{BB hardness for packing polytopes and set-cover}\label{sec:combinatorial}

{In this section, we} will begin by presenting a packing polytope with BBrank of $2^{\Omega(n)}$. The proof of this result will be based on a technique developed by Dadush and Tiwari~\cite{dadush2020complexity}. Then we will employ affine maps that satisfy Lemma \ref{lem:hardness} to obtain lower bounds on BBrank for a set-cover instance. 


We present a slightly generalized version of a key result from~\cite{dadush2020complexity}. The proof is essentially the same as of the original version, but we present it for completeness.

\begin{lemma}[Generalized Dadush-Tiwari Lemma]\label{lem:dadush}
Let $P \subseteq \R^n$ be an integer-infeasible non-empty polytope. Further, suppose $P$ is defined by the set of constraints $C_P$ (i.e. $P = \{x : C_P\}$) and let $D \subseteq C_P$ be a subset of constraints such that if we remove any constraint in $D$, the polytope becomes integer feasible (i.e. for all subsets $C \subset C_P$ such that $D \setminus C \not = \emptyset$, it holds that $\{x : C\} \cap {\Z^n} \not = \emptyset$). Then, any branch-and-bound tree $\mathcal{T}$ proving the integer-infeasibility of $P$ has at least $\frac{|D|}{n}$ leaf nodes, that is $|\T|  \geq 2\frac{|D|}{n}  - 1$.
\end{lemma}

\begin{proof}
Let $\mathcal{T}$ denote any branch-and-bound proof of infeasibility for $P$ and let $N$ denote the number of leaf nodes of $\mathcal{T}$. Suppose for sake of contradiction, that $N < \frac{|D|}{n}$. Consider any leaf node $v$ of $\mathcal{T}$. Let $C_v$ be the set of branching constraints on the path to $v$. Since $v$ is a leaf and $\mathcal{T}$ is a proof of infeasibility, we note $\{x : C_v \cup C_P\} = \{x : C_v\} \cap P = \emptyset$. 

By Helly's Theorem~\cite{barvinok}, there exists a set of $n+1$ constraints $K_v \subseteq C_v \cup C_P$ such that $\{x : K_v\} = \emptyset$. Also, we see that 
\begin{align}\label{eq:missone}
|K_v \cap C_P| \leq n. 
\end{align} 
This is because if we had $|K_v \cap C_P| = n+1$, this would imply $K_v \subseteq C_P$, hence $\{x : C_P\} \subseteq \{x : K_v\}$, and since $\{x : K_v\} = \emptyset$; this would imply $\{x : C_P\} = P = \emptyset$, which is clearly a contradiction because we know $P$ is non-empty. 

Next, observe that the set $\tilde{P}:= \left\{x: \bigcup_{v \in \text{leaves}(\mathcal{T})} (K_v \cap C_P) \right\}$ is integer-infeasible, because in fact $\mathcal{T}$ certifies this: since $\tilde{P} \cap \{x: C_u\}  = \left\{x: C_u \cup \bigcup_{v \in \text{leaves}(\mathcal{T})} (K_v \cap C_P) \right\} \subseteq \{x: K_u\} = \emptyset$ for all $u \in \text{leaves}(\mathcal{T})$. 
On other hand, observe that by (\ref{eq:missone}) we have that $|\bigcup_{v \in \text{leaves}(\mathcal{T})} (K_v \cap C_P)| \leq nN < |D|$, {so one of the inequalities in $D$ is not used in the description of $\tilde{P}$ and hence $\tilde{P}$} 
contains an integer point, a contradiction. This concludes the proof. 
\end{proof}

\subsection{Packing polytopes}\label{sec:packing}
Consider the following packing polytope
$$P_{PA} = \Big\{x \in [0,1]^n : \sum_{i \in S} x_i \leq k - 1 \text{ for all } S \subseteq [n] \textrm{ such that } |S| = k\Big\},$$
where we assume $2 \leq k \leq  \frac{n}{2}$. 
\begin{lemma}\label{lem:packing}
There exists an $x^* \in P_{PA} \setminus (P_{PA})_I $ such that any branch-and-bound tree that separates $x^*$ from $P_{PA}$ has at least $\frac{2}{n} \left(\binom{n}{k} + 1 \right) - 1$ nodes. Therefore, $\BBrank(P_{PA}) \geq \frac{2}{n} \left(\binom{n}{k} + 1 \right)  -1$.
\end{lemma}

The starting point for proving this lemma is the following following proposition. 

\begin{proposition}\label{prop:packing_infeas}
{The polytope $Q = P_{PA} \cap \{x : \ip{\ones}{x} \geq k\}$ is integer-infeasible}, and any branch-and-bound tree proving its {integer-}infeasibility has at least $\frac{2}{n} (\binom{n}{k} + 1)  - 1$ nodes. 
\end{proposition}

\begin{proof}
We show $Q \cap \{0,1\}^n = \emptyset$. Suppose for sake of contradiction there is some $x^* \in Q \cap \{0,1\}^n$. Since $\sum_{i \in [n]} x^*_i \geq k$, there is a set $S \subseteq [n]$ of size $k$ such that $\sum_{i \in S} x^*_i = k$. This violates the cardinality constraint corresponding to $S$, so $x^* \not \in Q$, a contradiction. 

{For the lower bound on BB trees that prove the integer-infeasibility of $Q$,} we show that $Q$ satisfies all of the requirements of Lemma \ref{lem:dadush}. First we show that $Q \neq \emptyset$. Consider the point $\hat{x} \in \mathbb{R}^n$ where $\hat{x}_i = \frac{k}{n}$ for $i \in [n]$.  Then, for any $S \subseteq [n]$ with $|S| = k$, we have $\sum_{i \in S}\hat{x}_i = k\cdot \frac{k}{n} \leq k \cdot \frac{1}{2} \leq k - 1$, where the last two inequalities are implied by the assumption $2 \leq k \leq  \frac{n}{2}$. Also, $\sum_{i \in [n]} \hat{x}_i = k$. Thus, $\hat{x}$ satisfies all the constraints of $Q$.

Next, we show that there is a set of $\binom{n}{k} + 1$ constraints $D$ such that removing any of these constraints makes $Q$ integer feasible.  Suppose we remove the constraint $\sum_{i \in S} x_i \leq k - 1$, denote this new polytope $Q'$. Then let $x^*_i = 1$ for all $i \in S$ and $x^*_i = 0$ for all $i \not \in S$. Clearly $\sum_{i \in [n]} x^*_i \geq k$ and since for all $S' \subseteq [n], |S'| = k$ it holds that $|S' \cap S| \leq k - 1$, it is also the case that $\sum_{i \in S'} x^*_i \leq k-1$. So $x^* \in Q' \cap \{0,1\}^n$. Now suppose we remove instead the constraint $\sum_{i \in [n]} x_i \geq k$, resulting in polytope $P_{PA}$. Clearly $P_{PA}$ is down monotone, and therefore $0 \in P_{PA}$. 

Therefore, by Lemma \ref{lem:dadush}, any branch-and-bound proof of infeasibility for $Q$ has at least $\frac{2}{n} (\binom{n}{k} + 1) - 1$ nodes.
\end{proof}

Now, combining Proposition \ref{prop:packing_infeas} with Lemma \ref{lem:inf_to_opt}, we are ready to prove Lemma \ref{lem:packing}. 

\begin{proof}[Proof of Lemma \ref{lem:packing}]
We will show that $P_{PA}$ and $\ip{\ones}{x} \le k-1$ satisfy the conditions on $P$ and $\ip{c}{x} \le \delta$ set by Lemma \ref{lem:inf_to_opt}. First, $\ip{\ones}{x} \le k-1$ is a valid inequality for $(P_{PA})_I$: this follows from the integer-infeasibility of $Q = P_{PA} \cap \{x : \ip{\ones}{x} \geq k\}$, as proven in Proposition \ref{prop:packing_infeas}. {Also, clearly $\ip{\ones}{x} \le k-1$ is not valid for $P_{PA}$, since it cuts off the point $\hat{x} = \frac{k}{n} \ones \in P_{PA}$.}
In the following paragraph we will show that $\{x \in (P_{PA})_I : \ip{\ones}{x} = k-1 \}$ has dimension $n-1$, that is, $\ip{\ones}{x} \le k-1$ is facet-defining for $(P_{PA})_I$. With this at hand we can apply Lemma \ref{lem:inf_to_opt} to obtain 
$$\BBrank(P_{PA}) \geq \BBhardness(P_{PA} \cap \{x : \ip{\ones}{x} \geq k\}) = \frac{2}{n} \left(\binom{n}{k} + 1\right) - 1$$
where the last inequality follows from Proposition \ref{prop:packing_infeas}, which will then prove the lemma.

To show that $\ip{\ones}{x} \le k-1$ is facet-defining, let $T \subseteq [n]$ be such that $|T| = k -1$. Let $\chi(T)$ denote the characteristic vector of $T$, so that $\chi(T)_i = 1$ if and only if $i \in T$. We know that all these points belong to the hyperplane $\{x: \ip{\ones}{x} = k - 1\}$. Thus, there can be at most $n$ affinely independent points among $\{\chi(T)\}_{T \subseteq [n], |T| = k -1}$. We first verify that there are exactly $n$ affinely independent points among $\{\chi(T)\}_{T \subseteq [n], |T| = k - 1}$ by showing that the affine hull of the points in $\{\chi(T)\}_{T \subseteq [n], |T| = k -1}$ is the hyperplane $\{x: \ip{\ones}{x} = k - 1\}$. Consider the system in variables $a,b$:
\begin{eqnarray*}
\ip{a}{\chi(T)} = b,~~~\forall T \subseteq [n] \textrm{ such that } |T| = k- 1. 
\end{eqnarray*} 
We have to show that all {non-zero} solutions of the above system are a scaling of $(\ones, k -1)$. For that, let $T^1 = \{1, \dots, k - 1\}$ and $T^2:= \{2, \dots, k \}$. Subtracting the equation corresponding to $T^1$ from that of $T^2$, we obtain $a_1 = a_{k}$. Using the same argument by suitably selecting $T^1$ and $T^2$, we obtain: $a_1 = a_2 = \dots = a_n$. Therefore, without loss of generality {(excluding the solution where $a$ is identically 0, since that would lead to $b = 0$)}, we may rescale all the $a_i$'s to $1$. Then we see the only possible value for $b$ is $k -1$. This shows that the only affine subspace containing the points $\{\chi(T)\}_{T \subseteq [n], |T| = k - 1}$ is $\{x: \ip{\ones}{x} = k - 1\}$, in other words, there are $n$ affinely independent points among them. This concludes the proof. 
\end{proof}

Finally, {since BB hardness is always at least the BB rank (Lemma \ref{lem:rankhardness}),} Lemma \ref{lem:packing} gives the desired hardness bound. 

\begin{cor}
Consider the polytope \mbox{$P_{PA} = \{x \in [0,1]^n : \sum_{i \in S} x_i \leq \frac{n}{2}, \text{ for all } S \subseteq [n] \textrm{ such that } |S| = \frac{n}{2} + 1\}$.} Then, ${\BBhardness}(P_{PA}) \geq 2^{\Omega(n)}$, i.e. there exists a {vector} $c\in \R^n$ such that  the smallest branch-and-bound tree that solves $$\max_{x \in P_{PA} \cap \{0,1\}^{n}} \ip{c}{x}$$ has size at least $2^{\Omega(n)}$.
\end{cor}

\subsection{Set-cover}

In order to obtain a set-cover instance that requires an exponential-size branch-and-bound tree, we will use Lemma~\ref{lem:hardness} together with  the flipping affine mapping (Defintion~\ref{defn:flip}) applied to the packing instance from Section~\ref{sec:packing}.  

More precisely, consider the following set-cover polytope: $$T_{\text{SC}} = \Big\{y \in [0,1]^n : \sum_{i \in S} y_i \geq 1 \text{ for all } S \subseteq [n] \textrm{ such that } |S| = k\Big\}.$$

\begin{proposition}
Let $P_{PA}$ still be the packing polytope from Section \ref{sec:packing}. Let $f : [0,1]^n \rightarrow [0,1]^n$ be the flipping function with $J = [n]$. Then:
\begin{itemize} \vspace{-10pt}
\item $T_{\text{SC}} = f(P_{PA})$
\item $T_{\text{SC}} \cap \{0,1\}^n \subseteq f(P_{PA} \cap \{0,1\}^n)$.
\end{itemize}
\end{proposition}
\begin{proof}
By substituting $y_i = 1 - x_i$ for  $i \in [n]$ in the polytope $P_{PA}$, we obtain that $T_{\text{SC}} = f(P_{PA})$.  

For the second item, consider any $
\hat{y} \in T_{\text{SC}} \cap \{0,1\}^n$. Notice $\hat{x} := 1 - \hat{y}$ belongs to $P_{PA} \cap \{0,1\}^n$ and $\hat{y} = f(\hat{x})$, and hence $\hat{y} \in f(P_{PA} \cap \{0,1\}^n)$. This gives $T_{\text{SC}} \cap \{0,1\}^n \subseteq f(P_{PA} \cap \{0,1\}^n)$.
\end{proof}



Then by Lemmas \ref{lem:hardness} {and \ref{lem:packing}} we have that $\BBrank(T_{\text{SC}}) \geq \BBrank(P_{PA}) \geq 2^{\Omega(n)}$. {Further employing Lemma~\ref{lem:rankhardness} we obtain the desired hardness bound.}

\begin{cor}
${\BBhardness}(T_{\text{SC}}) \geq 2^{\Omega(n)}$, i.e. there exists a {vector} $c\in \R^n$ such that  the smallest branch-and-bound tree that solves $$\max_{x \in T_{SC} \cap \{0,1\}^{n}} \ip{c}{x}$$ has size at least $2^{\Omega(n)}$.
\end{cor}

\section{BB hardness for cross-polytope}\label{sec:cross_poly_hard}

In this section, we present in Proposition~\ref{lem:lowerbndsimple} a simple proof of BB hardness for the cross-polytope. As mentioned before, this result slightly improves on the result that can be directly obtained by applying Lemma~\ref{lem:dadush} to the cross-polytope. 

Next in this section we develop Proposition \ref{prop:crosspoly_lb} that shows that there is a point in the cross polytope that is hard to separate using BB trees of small size. This allows us to use the machinery of Lemma \ref{lem:hardness} and a composition of the affine functions described in Section \ref{sec:tech_lemmas} to connect the BB hardness of TSP to that of the cross-polytope, which we do in Section \ref{sec:tsp}. 

The cross-polytope is defined as
$$P_n = \left\{x \in [0,1]^n : \sum_{i \in J}x_i + \sum_{i \not \in J}(1 - x_i) \geq \frac{1}{2} \quad \forall J \subseteq [n] \right\}.$$
{Recall that the cross-polytope is integer-infeasible: every 0/1 point $\hat{x} \in \{0,1\}^n$ is cut off by the inequality given by the set $J = \{ i \in [n] : \hat{x}_i = 0\}$.}

\begin{proposition}\label{lem:lowerbndsimple}
Let $\mathcal{T}$ be a BB tree for $P_n$ that certifies the integer-infeasibility of $P_n$. Then $|\mathcal{T}| \geq 2^{n +1} - 1$ (i.e. $\BBhardness(P_n) \geq 2^{n+1} - 1$).
\end{proposition}
\begin{proof}
In order to certify the integer-infeasibility of $P_n$, the atom of every leaf-node must be an empty set. We will verify that in order for the atom of a leaf $v$ to be empty, no more than one integer point is allowed to satisfy the branching constraints $C_v$ of $v$. This will complete the proof, since we then must have at least $2^n$ leaves. 

Consider any leaf $v$ of $\mathcal{T}$ such that two distinct integer points are feasible for its branching constraints. Then the average of these two points is a point in $\{0, 1, \frac{1}{2}\}^n$ with at least one component equal to $\frac{1}{2}$, which also satisfies the branching constraints. However, a point in $\{0, 1, \frac{1}{2}\}^n$ with at least one component equal to $\frac{1}{2}$  satisfies the constraints defining $P_n$. Thus the atom of the leaf $v$ is non-empty.  
\end{proof}

\begin{cor}\label{cor:lowdimpd}
Let $F\subseteq \mathbb{R}^n$ be a face of $P_n$ with dimension $d$. Then $\BBhardness(F) \geq 2^{d + 1} - 1.$
\end{cor}
\begin{proof}
Notice that $F$ is a copy of $P_d$ with $n -d$ components fixed to $0$ or $1$. Thus, {there exists an appropriate embedding affine transformation $f$ (Definition~\ref{defn:embed}) such that}
  $f(P_d) = F$. Also since $F \cap \mathbb{Z}^n = \emptyset$, we obtain that $f$, $P_d$ and $F$ satisfy all the conditions of Corollary~\ref{cor:affine_infeas}. Thus, $\BBhardness(F)  \geq \BBhardness(P_d) \geq  2^{d + 1} - 1$, where the last inequality follows from Proposition~\ref{lem:lowerbndsimple}.
\end{proof}

Next we show that the point $\frac{1}{2} \ones$  is hard to separate from $P_n$. For that we need a technical result that any halfspace that contains $\frac{1}{2}\ones$ must also contain a face of $[0,1]^n$ of dimension at least $\lfloor n/2 \rfloor$.  

\begin{lemma}\label{lem:contains_face}
Consider any $(\pi, \pi_0) \in \mathbb{R}^{n} \times \mathbb{R}$ such that $\ip{\pi}{\frac{1}{2}\ones} > \pi_0$. Let $G = \{x \in [0,1]^n : \ip{\pi}{x} > \pi_0\}$. There exists a face $F$ of $[0, 1]^n$ of dimension at least {$\lfloor n/2 \rfloor$} contained in $G$. 
\end{lemma}

{
\begin{proof}
	By assumption we have $\pi_0 < \frac{1}{2} \sum_{i=1}^n \pi_i$. First consider the case where the vector $\pi$ is non-negative. By renaming the coordinates we can further assume that $\pi_1 \ge \pi_2 \ge \ldots \ge \pi_n \ge 0$. Then the face $F = \{ x \in [0,1]^n : x_i = 1,~\forall i \le \lceil n/2 \rceil\}$ has the desired properties: it has dimension $n - \lceil n/2 \rceil = \lfloor n/2 \rfloor$, and any $\hat{x} \in F$ has $$\ip{\pi}{\hat{x}} \ge \sum_{i = 1}^{\lceil n/2 \rceil} \pi_i \ge \frac{1}{2} \sum_{i=1}^n \pi_i > \pi_0,$$ where the second inequality follows from the ordering of the coordinates of $\pi$, and hence $F$ is contained in $G$. 

	The case when $\pi$ is not non-negative can be reduced to the above case by flipping coordinates. More precisely, let $J$ be the set of coordinates $i$ where $\pi_i < 0$, and consider the coordinate flipping operation (Definition \ref{defn:flip}) $f : \R^n \rightarrow \R^n$ that flips all coordinates in $J$. Notice that $$f(G) = \Big\{ x \in [0,1]^n ~:~ \sum_{i \in J} -\pi_i x_i + \sum_{i \notin J} \pi_i x_i \le \pi_0 - \sum_{i \in J} \pi_i \Big\},$$ and that defining the vector $\pi'$ as $\pi'_i = -\pi_i$ for $i \in J$ and $\pi'_i = \pi_i$ for $i \notin J$ and $\pi'_0 := \pi_0 - \sum_{i \in J} \pi_i$ we still have $\ip{\pi'}{\frac{1}{2}\ones} > \pi'_0$. Since $\pi'$ is non-negative, the previous argument shows that $f(G)$ has a face $F$ of $[0,1]^n$ of desired dimension, and hence $f^{-1}(F) = f(F)$ is a desired face of $[0,1]^n$ contained in $G$. 
\end{proof}
}

\begin{proposition}\label{prop:crosspoly_lb}
{For every $n$ such that $\lfloor n/2 \rfloor > 1$, $\BBdepth \left(\frac{1}{2}\ones, P_n\right) \geq 2^{\lfloor n/2\rfloor  + 1} - 1$.}
\end{proposition}

\begin{proof}
For sake of contradiction suppose there exists a tree $\mathcal{T}$ of size less than {$2^{\lfloor n/2\rfloor + 1 } - 1$} such that $\frac{1}{2}\ones \not \in \conv(\T(P_n))$. By the hyperplane separation theorem, there exists $(\pi, \pi_0) \in \mathbb{R}^{n} \times \mathbb{R}$ such that $\ip{\pi}{\frac{1}{2}\ones} > \pi_0$ and $\ip{\pi}{x} \leq \pi_0$ for all $x \in \conv(\T(P_n))$. By Lemma \ref{lem:contains_face}, let $F$ be a face of $[0,1]^n$ of dimension {$\lfloor n/2 \rfloor$} contained in $\{x \in \mathbb{R}^n \,|\, \langle \pi, x\rangle > \pi_0 \}$; notice that $P_n \cap F$ is a face of $P_n$ of the same dimension.
Since $\T(P_n) \subseteq \textup{conv}(\T(P_n)) \subseteq \mathbb{R}^n\setminus \{x \in \mathbb{R}^n \,|\, \langle \pi, x\rangle > \pi_0 \} \subseteq \mathbb{R}^n \setminus F$  and $\T(F) \subseteq F$, we have that $\T(P_n)$ and $\T(F)$ are disjoint and hence from Lemma \ref{lem:monotonicity} we get $\T(P_n \cap F) \subseteq \T(P_n) \cap \T(F) = \emptyset$, 
i.e., the atoms of the leaves of $\mathcal{T}$ applied to $P_n \cap F$ are all empty. Thus, $\mathcal{T}$ is a branch-and-bound tree to certify the integer-infeasibility of $P_n \cap F$
 of size less than {$2^{\lfloor n/2 \rfloor + 1} - 1$}.  However, this contradicts Corollary~\ref{cor:lowdimpd}.

\end{proof}


\section{BB hardness for perturbed cross-polytope}\label{sec:perturbed_cp}

We now show that exponential BB hardness for the cross-polytope persists even after adding Gaussian noise to the entries of the contraint matrix. 
 This implies an exponential lower bound even for a ``smoothed analysis" of general branch-and-bound.

	We consider the cross-polytope $P_n$ but where we add an independent gaussian noise $N(0,1/20^2)$ with mean 0 and variance $1/20^2$ to each coefficient in the left-hand side of its defining inequalities, and replace the {right-hand sides} by 
{approximately}
$\frac{n}{20}$ instead of the traditional $\frac{1}{2}$. More precisely, we consider the following random polytope $Q$:
	\begin{align*}
		Q ~:=~ \left\{x \in [0,1]^n ~:~ \sum_{i \in I} \left(1 + N\left(0,\tfrac{1}{20^2}\right)\right) x_i + \sum_{i \notin I} \Big(1 - \left(1 + N \left(0,\tfrac{1}{20^2}\right)\right) x_i\Big) \ge \frac{1.6 n}{20}, ~~~~\forall I \subseteq [n]\right\}
	\end{align*}
	where each occurrence of $N(0, \frac{1}{20^2})$ is independent. 
	
	\begin{thm} \label{thm:main}
		With probability at least $1 - \frac{2}{e^{n/2}}$ the polytope $Q$ is integer-infeasible and every BB tree proving its infeasibility has at least $2^{\Omega(n)}$ nodes. 
	\end{thm}
	
	{Recall that for independent gaussians $Y \sim N(\mu,\sigma^2)$ and $Y' \sim N(\mu', (\sigma')^2)$, their sum $Y+Y'$ is distributed as $N(\mu+\mu',\sigma^2 + (\sigma')^2)$, and for a centered gaussian $Y \sim N(0,\sigma^2)$ the scaled random variable $\alpha Y$ is distributed as $N(0, \alpha^2 \sigma^2)$ for all $\alpha \in \R$.}
	
	We need the following standard tail bound for the Normal distribution (see equation (2.10) of \cite{vershynin2018high}).
	
	\begin{fact} \label{fact:gauss}
		Let $X \sim N(0, \sigma^2)$ be a mean zero gaussian with variance $\sigma^2$. Then for every $p \in (0,1)$, with probability at least $1-p$ we have $X \le \sigma \sqrt{2 \ln(1/p)}$, and with probability at least $1-p$ we have $X \ge - \sigma \sqrt{2 \ln (1/p)}$.
	\end{fact}
	
	Let $LHS_I(x)$ be the {left-hand-side }
of the constraint of $Q$ indexed by the set $I$ evaluated at the point $x$.

	\begin{lemma} \label{lemma:inf}
		With probability at least $1 - \frac{1}{e^{n/2}}$ the polytope $Q$ is integer-infeasible. 
	\end{lemma}
	
	\begin{proof}
		Fix a 0/1 point $x \in \{0,1\}^n$, and let $I \subseteq [n]$ be the set of coordinates $i$ where $x_i = 0$. Let $I^c = [n]\setminus I$. Notice $LHS_I(x)$ is a gaussian random variable with mean 0 and variance $\frac{|I^c|}{20^2} \le \frac{n}{20^2}$, and so from Fact \ref{fact:gauss}, with probability at least $1 - \frac{1}{e^{n/2} 2^n}$ we have $$LHS_I(x) \le \frac{\sqrt{n}}{20} \sqrt{2 \ln(e^{n/2} 2^n)}= \frac{\sqrt{n}}{20} \sqrt{(1 + 2 \ln 2) n} < \frac{1.6n}{20},$$ i.e., the point $x$ does not satisfy the inequality of $Q$ indexed by $I$, and so does not belong to $Q$. Taking a union bound over all $2^n$ points $x \in \{0,1\}^n$, with probability at least $1 - \frac{1}{e^{n/2}}$ none of them belong to $Q$. This concludes the proof. 
	\end{proof}
	
	\begin{lemma} \label{lemma:half}
		With probability at least $1 - \frac{1}{e^{n/2}}$ the polytope $Q$ contains all points $\{0,\frac{1}{2},1\}^n$ that have at least $s = \frac{4n}{10}$ coordinates with value $\frac{1}{2}$. (We call this set of points $\textup{Half}_s$.)
	\end{lemma}
	
	\begin{proof}
		Consider $x \in \textup{Half}_s$. Fix $I \subseteq [n]$. Let $n_{\textrm{half}} \ge s = \frac{4n}{10}$ be the number of coordinates of $x$ with value $\frac{1}{2}$, $n_{\textrm{ones}}$ be the number of coordinates with value 1, and let $n_{\textrm{diff}}$ be the number of coordinates $i$ where either $i \in I$ and $x_i = 1$, or $i \notin I$ and $x_i = 0$. We see that $LHS_I(x)$ distributed as 
		\begin{align*}
			LHS_I(x) &=_{d} \frac{n_{\textrm{half}}}{2} + n_{\textrm{diff}} + \frac{1}{2} N(0, \tfrac{n_\text{half}}{20^2}) + N(0, \tfrac{n_{\textrm{ones}}}{20^2})\\
			&=_{d} \frac{n_{\textrm{half}}}{2} + n_{\textrm{diff}} + N(0, (\tfrac{n_\text{half}}{4} + n_{\textrm{ones}}) \cdot \tfrac{1}{20^2}),
		\end{align*}
		where again the occurrences of $N(0,\cdot)$ are independent. Since the last term is a gaussian with variance at most $\frac{n}{20^2}$, we get that with probability at least $1- \frac{1}{e^{n/2} \cdot 2^n \cdot 3^n}$ 
		\begin{align*}
			LHS_I(x) \ge \frac{n_{\textrm{half}}}{2} + n_{\textrm{diff}} - \frac{1}{20} \sqrt{n} \sqrt{2 \log(e^{n/2} \cdot 2^n \cdot 3^n)} \ge \frac{4n}{20} - \frac{2.4n}{20} = \frac{1.6n}{20},
		\end{align*}
	that is, $x$ satisfies the constraint of $Q$ indexed by $I$. 
	
	Taking a union bound over all $x \in \textup{Half}_s$ and all subsets $I \subseteq [n]$, we see that all points in $\textup{Half}_s$ satisfy all constraints of $Q$ with probability at least $1 - \frac{1}{e^{n/2}}$. This concludes the proof. 
	\end{proof}
	
	\begin{lemma} \label{lemma:shatter}
		Let $F \subseteq \{0,1\}^n$ be a set of 0/1 points. For any $k$, if $|F| > \sum_{i \le k-1} \binom{n}{i}$, then $\conv(F)$ contains a point with at least $k$ coordinates of value 1/2. 
	\end{lemma}
	
	\begin{proof}
		By the Sauer-Shelah Lemma (Lemma 11.1 of \cite{jukna}), there is a set of coordinates $J \subseteq [n]$ of size $|J| = k$ such that the points in $F$ take all possible values in coordinates $J$, i.e., the projection $F_J$ onto the coordinates $J$ equals $\{0,1\}^k$. So the point $\frac{1}{2} \ones$ belongs to $\conv(F_J)$, which implies that $\conv(F)$ has the desired point. 
	\end{proof}

	\begin{proof}[Proof of Theorem \ref{thm:main}]
		Let $E$ be the event that both the bounds from Lemmas \ref{lemma:inf} and \ref{lemma:half} hold. By a union bound, this event happens with probability at least $1 - \frac{2}{e^{n/2}}$. So it suffices to show that there is a constant $c > 0$ such that for every scenario in $E$, every BB tree proving the infeasibility of $Q$ has at least $2^{cn}$ leaves. 

		In hindsight, again let $s = \frac{4n}{10}$ and set $c := 1 - h(\frac{s}{n})$, where $h$ is the binary entropy function $h(p) := p \log \tfrac{1}{p} + (1-p) \log \tfrac{1}{1-p}$. Notice that $c > 0$, since $h$ is strictly increasing in the interval $[0,\frac{1}{2}]$ and hence $h(\frac{s}{n}) < h(\frac{1}{2}) = 1$. 
		
	{Fix a scenario in the event $E$, so we know $Q$ is integer-infeasible and $\textup{Half}_s \subseteq Q$.} Consider any tree $\T$ that proves integer-infeasibility of $Q$, and we claim that it has more than $2^{cn}$ leaves. By contradiction, suppose not. Then $\T$ has a leaf $v$ whose branching constraints $C_v$ are satisfied by at least $\frac{2^n}{2^{cn}} = 2^{n\cdot h(s/n)}$ 0/1 points { (recall that each integer point satisfies all of the branching constraints of at least some leaf)}. But since $2^{n\cdot h(s/n)} > \sum_{i \le s-1}\binom{n}{i}$ (see e.g. Lemma 5 of \cite{gottlieb}), by Lemma \ref{lemma:shatter} we know that the convex set $\{x : C_v\}$ contains a point $\hat{x} \in [0,1]^n$ with at least $s$ coordinates of value $1/2$. Moreover, notice that $\hat{x}$ also belongs to $\conv(\textup{Half}_s)$, which is contained in $Q$. Hence $\hat{x} \in \{x : C_v\} \cap Q$, namely the atom of the leaf $v$. But this contradicts that this atom is empty (which is required since $\T$ proves integer infeasibility of $Q$). This concludes the proof. 
	\end{proof}

\section{BB hardness for TSP}\label{sec:tsp}


{Again we use $P_k$ to denote the cross-polytope in $k$ dimensions.}

\begin{proposition}\label{prop:tsp_hard}
Let $f$ be any composition of the flipping (Defintion~\ref{defn:flip}), embedding (Definition~\ref{defn:embed}), and duplication (Definition~\ref{defn:dup}) operations. Let $H \subseteq [0,1]^n$ be a polytope such that $f(P_k) \subseteq H$ and $f(\frac{1}{2}\ones) \not \in H_I$, where $k \leq n$. Then, {$\BBrank(H) \geq 2^{\lfloor k/2 \rfloor}$}. 
\end{proposition}

\begin{proof}
Notice that if $f$ is a composition of the flipping, embedding, and duplication operations, then $f$ is an integral affine transformation. 
Moreover, if $P$ is integer-infeasible then $f(P)$ is also integer-infeasible. In particular, since $P_k$ is integer-infeasible we have $(f(P_k))_I = \emptyset$, and hence $f(\frac{1}{2}\ones) \not \in (f(P_k))_I$.  Now, Corollary \ref{cor:affine_depth} and Proposition \ref{prop:crosspoly_lb} give us {$\BBdepth \left(f \left(\frac{1}{2}\ones \right), f(P_k)\right) \geq \BBdepth \left(\frac{1}{2}\ones, P_k \right) \geq 2^{\lfloor k/2 \rfloor +1} - 1$}. Finally, since $f(P_k) \subseteq H$, Corollary \ref{cor:monoton_depth} implies that {$\BBdepth(f(\frac{1}{2}\ones), H) \geq 2^{\lfloor k/2 \rfloor + 1} - 1$}. This implies the desired result: {$\BBrank(H) \geq 2^{\lfloor k/2 \rfloor + 1} - 1 \ge 2^{\lfloor k/2 \rfloor}$}. 
\end{proof}

We next state a key result from { the proof of Theorem 4.1} of~\cite{cook2001matrix} (see also~\cite{chvatal1989cutting}), that shows how we can apply Proposition~\ref{prop:tsp_hard} to obtain BB hardenss of the TSP polytope.  Let $T_{\text{TSP}_n}$ be the {LP relaxation of the TSP polytope using subtour elimination constraints} for $n$ cities:
\begin{align*}
    x(\delta(v)) = 2 & \quad \forall v \in V \\
    x(\delta(W)) \geq 2 & \quad \forall W \subset V, W \not = \emptyset \\
    0 \leq x(e) \leq 1 & \quad \forall e \in E
\end{align*}

\begin{proposition}[{proof of Theorem 4.1 in \cite{cook2001matrix}}]
There exists a function $f$ which is a composition of flipping, embedding, and duplication such that $f(P_{\lfloor n/8 \rfloor})$ is contained in $T_{\text{TSP}_n}$ and $f(\frac{1}{2}\ones)$ does not belong to the integer hull of $T_{\text{TSP}_n}$. 
\end{proposition}

Then employing Proposition~\ref{prop:tsp_hard} we obtain {$\BBrank(T_{\text{TSP}_n}) \ge 2^{\frac{n}{16} - 2}$, and again since BB hardness is at least the BB rank (Lemma \ref{lem:rankhardness}) we obtain the desired hardness.} 

\begin{cor}[BB hardness for TSP]
{$\BBhardness(T_{\text{TSP}_n)} \geq  2^{\frac{n}{16} - 2}$}, i.e, there is a {vector} $c \in \R^{n(n -1)/2}$ such that the smallest branch-and-bound tree that solves $$\max_{x \in T_{\text{TSP}_n} \cap \{0,1\}^{n(n -1)/2}} \ip{c}{x}$$ has size at least {$2^{\frac{n}{16} - 2}$}.
\end{cor}

\bibliography{bib}
\bibliographystyle{abbrv}

\end{document}